\documentclass[11pt,english]{article}

\usepackage[margin = 2.5cm]{geometry}

\usepackage{amsmath}
\usepackage{amsthm}
\usepackage{amssymb}
\usepackage{mathtools}
\usepackage{graphicx}
\usepackage{xcolor}
\usepackage[hidelinks, bookmarks = false]{hyperref}
\usepackage{enumitem}

\usepackage{caption}
\usepackage{cleveref}
\theoremstyle{plain}

\newtheorem{theorem}{Theorem}[section]
\crefname{theorem}{Theorem}{Theorems}
\Crefname{theorem}{Theorem}{Theorems}

\newtheorem*{lemma*}{Lemma}
\newtheorem{lemma}[theorem]{Lemma}
\crefname{lemma}{Lemma}{Lemmas}
\Crefname{lemma}{Lemma}{Lemmas}

\newtheorem*{claim*}{Claim}

\crefname{claim}{Claim}{Claims}
\Crefname{claim}{Claim}{Claims}

\crefname{proposition}{Proposition}{Propositions}
\Crefname{proposition}{Proposition}{Propositions}

\newtheorem{corollary}[theorem]{Corollary}
\crefname{corollary}{Corollary}{Corollaries}
\Crefname{corollary}{Corollary}{Corollaries}

\newtheorem{conjecture}[theorem]{Conjecture}
\crefname{conjecture}{Conjecture}{Conjectures}
\Crefname{conjecture}{Conjecture}{Conjectures}

\crefname{question}{Question}{Questions}
\Crefname{question}{Question}{Questions}

\crefname{observation}{Observation}{Observations}
\Crefname{observation}{Observation}{Observations}

\crefname{example}{Example}{Examples}
\Crefname{example}{Example}{Examples}

\theoremstyle{definition}

\crefname{problem}{Problem}{Problems}
\Crefname{problem}{Problem}{Problems}

\newtheorem{definition}[theorem]{Definition}
\crefname{definition}{Definition}{Definitions}
\Crefname{definition}{Definition}{Definitions}

\usepackage{xpatch}
\xpatchcmd{\proof}{\itshape}{\normalfont\proofnamefont}{}{}
\newcommand{\proofnamefont}{}
\renewcommand{\proofnamefont}{\bfseries}

\usepackage{framed}

\newcommand{\remove}[1]{}

\DeclareMathOperator{\ex}{ex}

\newcommand{\C}{\mathcal{C}}

\newcommand{\al}{\alpha}
\newcommand{\A}{\mathcal{A}}
\newcommand{\B}{\mathcal{B}}

\newcommand{\F}{\mathcal{F}}

\usepackage{setspace}
\title{Tree-degenerate graphs and nested dependent random choice}

\author{
        Tao Jiang 
        \thanks{Department of Mathematics, Miami University, Oxford, OH 45056, USA. Email: \texttt{jiangt}@\texttt{miamioh.edu}.
        Research supported by National Science Foundation grant DMS-1855542.}
    \and
	    Sean Longbrake\thanks{
	    Department of Mathematics, Miami University, Oxford, OH 45056, USA. Email: \texttt{
	    longbrsa}@\texttt{miamioh.edu}.
        Research partially supported by National Science Foundation grant DMS-1855542.}	
}

\begin{document}

\maketitle

\begin{abstract}

	\setlength{\parskip}{\medskipamount}
    \setlength{\parindent}{0pt}
    \noindent
    The celebrated dependent random choice lemma states that in a bipartite graph
    an average vertex (weighted by its degree) has the property that almost all small subsets
    $S$ in its neighborhood has common neighborhood almost as large as in the random graph of the
    same edge-density. Two well-known applications of the lemma are as follows.
    The first is a theorem of F\"uredi \cite{Furedi} 
    and of Alon, Krivelevich, and Sudakov \cite{AKS} showing that the maximum number of edges in an $n$-vertex graph 
    not containing a fixed bipartite graph with maximum degree at most $r$ on one side
    is $O(n^{2-1/r})$. This was recently extended by Grzesik, Janzer and Nagy \cite{GJN} 
    to the family of so-called $(r,t)$-blowups of a tree. A second application is a theorem of Conlon, Fox, and Sudakov \cite{CFS},
    confirming a special case of a conjecture of Erd\H{o}s and Simonovits and of Sidorenko, showing that if $H$ is a bipartite graph that contains a vertex complete to the other part and $G$ is a graph 
    then the probability that
    the uniform random mapping from $V(H)$ to $V(G)$ is a homomorphism
    is at least $\left[\frac{2|E(G)|}{|V(G)|^2}\right]^{|E(H)|}$.

    In this note, we introduce a nested variant of the dependent random choice lemma, which might be
    of independent interest. We then apply it  to
    obtain a common extension of the theorem of Conlon, Fox, and Sudakov and the theorem of Grzesik, Janzer, and Nagy,
    regarding Tur\'an and Sidorenko properties of so-called tree-degenerate graphs. 
 \end{abstract}

\section{Introduction}
Given a graph $G$, let $|G|$ denote its number of vertices.
A {\it homomorphism} from a graph $H$ to a graph $G$ is a mapping $f:V(H)\to V(G)$ such that for each edge $uv$ in $H$, $f(u)f(v)$ is an edge in $G$. Let $h_H(G)$ denote the number of homomorphisms from $H$ to $G$ and $t_H(G)=h_H(G)/|G|^{|H|}$. Thus $t_H(G)$ represents
the fraction of mappings from $V(H)$ to $V(G)$ that are homomorphisms. Viewed probabilistically, $t_H(G)$ is the 
probability that the uniform random mapping from $V(H)$ to $V(G)$ is a homomorphism.
A beautiful conjecture of Sidorenko \cite{Sidorenko} is as follows.

\begin{conjecture}[Sidorenko\cite{Sido-old}] \label{conj:sidorenko}
For every bipartite graph $H$ and every graph $G$, 
\[t_H(G)\geq [t_{K_2}(G)]^{e(H)}.\]
\end{conjecture}

Since $[t_{K_2}(G)]^{e(H)}=\left(\frac{2e(G)}{n^2}\right)^{e(H)}$, one may view Sidorenko's conjecture as saying that the number of homomorphic copies of $H$ in an $n$-vertex graph $G$ is asymptotically
at least as large as in the $n$-vertex random graph with the same edge-density.
The following lemma, based on tensor products, 
(see Remark 2 in the English version of \cite{Sido-old} for instance), 
is commonly known and is used in many earlier papers
(see \cite{AR},\cite{CFS},\cite{KLL} for instance). It reduces the conjecture
to a slightly weaker statement.
\begin{lemma}[\cite{Sido-old}]\label{lem:reduction}
Let $H$ be a bipartite graph. If there exists a positive constant $c$
depending only on $H$ such that for all graphs $G$ $t_H(G)\geq c[t_{K_2}(G)]^{e(H)}$
holds, then for all $G$, $t_H(G)\geq [t_{K_2}(G)]^{e(H)}$.
\end{lemma} 
When the edge-density of $G$ is sufficiently high, it is expected
that many of the homomorphisms from $H$ to $G$ are injective. 
Erd\H{o}s and Simonovits \cite{ES-cube} made several conjectures regarding the
number of injective homomorphisms. As usual, let $\ex(n,H)$ denote the Tur\'an number of $H$, 
which is the maximum number of edges in an $n$-vertex graph not containing $H$.
Let $h^*_H(G)$ denote the number of injective homomorphisms from $H$ to $G$,
and let $t^*_H(G)=h^*_H(G)/|G|^{H|}$. The first conjecture of
Erd\H{o}s and Simonovits from \cite{ES-cube} states that
for every $c>0$ there is a $c'>0$ such that if
$e(G)>(1+c)\ex(n,H)$ then $t^*_H(G)\geq c' t_{K_2}(G)^{e(H)}$.
The second, weaker, conjecture from \cite{ES-cube} says
if $\ex(n,H)=O(n^{2-\al})$ then there exist constants $0\leq \tilde{\al}\leq \alpha, c,c'>0$ such that if $e(G)>cn^{2-\tilde{\al}}$ then $t^*_H(G)\geq c' t_{K_2}(G)^{e(H)}$.
It is known (see \cite{Sido-old})
that this weaker version is equivalent to Sidorenko's conjecture.
However, compared to the stronger conjecture of Erd\H{o}s and Simonovits,
Sidorenko's conjecture does not give an explicit sharp
edge-density threshold on when to guarantee the stated number of injective homomorphisms.
There is yet another version of the Erd\H{o}s-Simonovits conjecture, given in 
\cite{Simonovits}, that is equivalent to saying that
there exist two constants $c,c'>0$ such that 
if $G$ is an $n$-vertex graph $G$ with $e(G)>c\ex(n,H)$ then $t^*_H(G)\geq c'[t_{K_2}(G)]^{e(H)}$.

Sidorenko \cite{Sidorenko} verified his own conjecture when $H$ is a complete
bipartite graph, an even cycle, a tree, or a bipartite graph with at most four vertices on one side.
Hatami \cite{Hatami} proved that hypercubes satisfy Sidorenko's conjecture by developing a concept of norming graphs.
The first major progress on Sidorenko's conjecture was made by Conlon, Fox and Sudakov \cite{CFS},
who used the celebrated dependent random choice method (see \cite{FS} for a survey) to show

\begin{theorem}[Conlon-Fox-Sudakov \cite{CFS}] \label{thm:CFS}
If $H$ is a bipartite graph with $m$ edges which has a vertex complete to the other part,
then $H$ satisfies Sidorenko's conjecture. 
\end{theorem}
In fact, Conlon, Fox and Sudakov proved the stronger theorem that $H$ contains a vertex complete
to the other part and  the minimum degree in the first
part is at least $d$ then $t_H(G)\geq [t_{K_{r,d}}(G)]^{\frac{m}{rd}}$.
From Theorem~\ref{thm:CFS}, Conlon, Fox, and Sudakov \cite{CFS} also deduced an approximate version of Sidorenko's conjecture.
Since the work of Conlon, Fox, and Sudakov, there has been a lot of further progress on Sidorenko's conjecture. Li and Szegedy \cite{LS} used
entropy method (presented in the form of logarithmic convexity inequalities) to the extend the result of Conlon, Fox and Sudakov to a more general family of graphs $H$, which they refer to as reflection trees.
These ideas were further developed by Kim, Lee and Lee \cite{KLL}, who proved the conjecture for 
what they called tree-arrangeable graphs and showed that if $T$ is a tree
and $H$ is a bipartite graph that satisfies Sidorenko's conjecture then
the Cartesian product of $T$ and $H$ also satisfies Sidorenko's conjecture.
Subsequently, Conlon, Kim, Lee and Lee \cite{CKLL,CKLL2}
and independently Szegedy \cite{szegedy} established more families of bipartite graphs $H$ for which
Sidorenko's conjecture holds. These include bipartite graphs that admit a certain type of tree decomposition,
subdivisions of certain graphs including cliques, and certain cartesian products, and etc. More recently, Conlon and Lee \cite{CL} showed that
Sidorenko's conjecture holds for any bipartite graph $H$ with a bipartition
$(A,B)$ where the number of vertices in $B$ of degree $k$ satisfies a certain divisibility condition for each $k$. As a corollary, for every bipartite graph $H$ with a bipartition $(A,B)$ there is a positive integer $p$ such that
the blowup $H_A^p$ formed by taking $p$ vertex-disjoint copies of $H$ and gluing
all copies of $A$ along corresponding vertices satisfies Sidorenko's conjecture.

Another line of work that motivates our  result is related to a long-standing conjecture of
Erd\H{o}s regarding the Tur\'an number of so-called $r$-degenerate graphs.
Given a positive integer $r$, a graph $H$ is {\it $r$-degenerate} if its
vertices can be linearly ordered such that each vertex has back degree at most $r$.  

\begin{conjecture}[Erd\H{o}s \cite{Erdos-1967}] \label{degenerate-conjecture}
Let $r$ be a fixed positive integer. Let $H$ be any $r$-degenerate
bipartite graph. Then $\ex(n,H)=O(n^{2-1/r})$.
\end{conjecture}

The first major progress on Conjecture~\ref{degenerate-conjecture} was the following theorem, which
was first obtained by F\"uredi \cite{Furedi} in an implicit form and then later
reproved by Alon, Krivelevich, and Sudakov \cite{AKS} using the 
dependent random choice method. 
\begin{theorem}[F\"uredi \cite{Furedi}, Alon-Krivelevich-Sudakov \cite{AKS}]
\label{thm:one-side}
Let $r$ be a positive integer.
Let $H$ be a bipartite graph with maximum degree 
at most $r$ on one side.
Then $\ex(n,H)=O(n^{2-1/r})$.
\end{theorem}
The family of graphs satisfying the condition of Theorem~\ref{thm:one-side}
forms a very special family of $r$-degenerate bipartite graphs, which we will refer
to as {\it one side $r$-bounded} bipartite graphs.
Recently, Grzesik, Janzer, and Nagy \cite{GJN}, among other things, extended Theorem~\ref{thm:one-side} to a broader family of graphs, called $(r,t)$-blowups of a tree.

\begin{definition}[$(r,t)$-blowups of a tree]\label{dfn:rt-blowup}
Let $r\leq t$ and $m$ be positive integers.
A bipartite graph $H$ is {\it an $(r,t)$-blowup of a tree} (or {\it $(r,t)$-blowup} in short)
with root block $B_0$ and non-root blocks $B_1,\dots, B_m$ if $B_0,B_1,\dots, B_m$ partition $V(H)$,
$|B_0|=r, |B_1|=\dots=|B_m|=t$
 and $H$ can be constructed by  joining $B_1$ completely to $B_0$ and for each $2\leq i\leq m$ joining $B_i$ completely to a $r$-subset 
 of $B_{\gamma(i)}$ for some $\gamma(i)\leq i-1$.
\end{definition}

\begin{theorem} [Grzesik-Janzer-Nagy \cite{GJN}] \label{thm:GJN}
Let $r\leq t$ be positive integers. If $H$ is  an $(r,t)$-blowup of a tree,
then $\ex(n,H)=O(n^{2-1/r})$.
\end{theorem}
Since every one-side $r$-bounded graph is a subgraph of an $(r,t)$-blowup with two blocks $B_0,B_1$,
Theorem \ref{thm:GJN} substantially generalizes Theorem~\ref{thm:one-side}.

In this paper, we give a common strengthening of Theorem~\ref{thm:CFS} and Theorem~\ref{thm:GJN}
by proving a general theorem on the Tur\'an and Sidorenko properties of 
so-called tree-degenerate graphs.

\begin{definition}[Tree-degenerate graphs]
A bipartite graph $H$ is {\it tree-degenerate} with root block $B_0$ and non-root blocks $B_1,\dots, B_m$ if $B_0,B_1,\dots, B_m$ partition $V(H)$
 and $H$ can be constructed by letting $P(B_1)=B_0$ and joining $B_1$ completely to $B_0$ and for each $2\leq i\leq m$ joining $B_i$ completely to a subset $P(B_i)$ 
 of $B_{\gamma(i)}$ for some $1\leq \gamma(i)\leq i-1$, such that 
 for all $i\geq 2$ $|P(B_{\gamma(i)})|\leq |P(B_i)|$.
 We call $P(B_i)$ the {\it parent set} of $B_i$ and $B_{\gamma(i)}$ the {\it parent block} of $B_i$
 and we call $P$ the {\it parent function}.
 We call $(B_0,\dots, B_m,P)$ a {\it block representation} of $H$.
\end{definition}

We present our main result in terms of so-called $r$-norm density.
We will explain the advantage of doing so after the presenting the theorem.
Let $G$ be a graph with $n$ vertices. For each positive integer $r$,
we define the the {\it $r$-norm density} of $G$, denoted by $p_r(G)$, as 
\[p_r(G):=t_{K_{1,r}}(G)^{1/r}.\]

Note that $p_1(G)=t_{K_2}(G)=\frac{2e(G)}{n^2}$,  is the usual edge-density of $G$.
In general, one may view $p_r(G)$ as a modified measure of edge-density of $G$ that takes the degree distribution into account. Using convexity, one can show that $p_r(G)\geq p_s(G)$ whenever $r\geq s$ (see Lemma~\ref{lem:monotone}).

\begin{theorem}[Main theorem] \label{thm:main}
For any graph $G$ and positive integer $\ell$, let $p_\ell(G)=t_{K_{1,\ell}}(G)^{1/\ell}$.
Let $H$ be a tree-degenerate graph with a block representation $(B_0,B_1,\dots, B_m,P)$.
Let $s=|B_0|$ and $r=\max_i|P(B_i)|$.
There exist  positive constants $c_1=c_1(H), c_2=c_2(H), c_3=c_3(H)$ depending only on $H$ such 
that for any graph $G$ 
\[t_H(G)\geq c_1 [p_s(G)]^{e(H)}\geq c_1[t_{K_2}(G)]^{e(H)}.\]
Furthermore, if  $h_{K_{1,r}}(G)>c_2n^r$, where $n=|G|$, then 
\[t^*_H(G) \geq c_3 [p_r(G)]^{e(H)}\geq c_3 [t_{K_2}(G)]^{e(H)}.\]
\end{theorem}

The first part of the theorem and Lemma~\ref{lem:reduction} imply
the following.
\begin{corollary} \label{cor:tree-degenerate}
Let $H$ be a tree-degenerate graph. Then $H$ satisfies Sidorenko's conjecture.
\end{corollary}
Since a bipartite graph $H$ containing a vertex complete to the other part
is tree-degenerate with $|B_0|=1$ and $\gamma(i)=1$ for all $i\geq 2$, Corollary~\ref{cor:tree-degenerate}
generalizes Theorem~\ref{thm:CFS}.

A special case of
the second part of Theorem~\ref{thm:main} yields the following.
\begin{corollary} \label{cor:(r,t)-blowup}
Let $r\leq t$ be positive integers. Let $H$ be an $(r,t)$-blowup of a tree with $h$ vertices. 
Then $\ex(n,H)=O(n^{2-1/r})$. Furthermore,
there exist constants $c,c'>0$ such that every $n$-vertex graph $G$ with $h_{K_{1,r}}(G)\geq cn^r$ contains at least
$c' n^h(\frac{2e(G)}{n^2})^{e(H)}$ copies of $H$.
\end{corollary}
Corollary~\ref{cor:(r,t)-blowup} strengthens Theorem~\ref{thm:GJN} in two ways.
First, it relaxes the density requirement on $G$ 
from $e(G)=\Omega(n^{2-1/r})$ to $h_{K_{1,r}}(G)=\Omega(n^r)$ (i.e. from $p_1(G)=\Omega(n^{-1/r})$
to $p_r(G)=\Omega(n^{-1/r})$). Second, it not only gives at least one copy of $H$, but an
optimal number (up to a multiplicative constant) of copies of $H$.
A closer examination of  the proof of Theorem~\ref{thm:GJN} given by Grzesik, Janzer, and Nagy in \cite{GJN} 
shows that their proof can be strengthened to also give Corollary~\ref{cor:(r,t)-blowup}.
However, Theorem~\ref{thm:main} is more general than Corollary~\ref{cor:(r,t)-blowup}, 
as the counting statement applies to any tree-degenerate graph $H$, where parent set sizes can vary, instead of just to $(r,t)$-blowups.
The relaxation of $p_1(G)=\Omega(n^{-1/r})$ to $p_r(G)=\Omega(n^{-1/r})$ is also a useful feature, as 
in bipartite Tur\'an problems  sometimes we need to handle cases where the host graph has very uneven degree distribution and hence
high $r$-norm density, despite having relatively low $1$-norm density (see \cite{JMN} for an instance of this kind).

To prove Theorem~\ref{thm:main}, we introduce a notion of goodness and prove a lemma that might be viewed as a 
nested variant of the dependent random choice lemma. Once we establish the lemma, the proof of Theorem~\ref{thm:main} readily follows. Conceivably this variant could find more applications.

We organize our paper as follows. 
In Section \ref{sec:lemmas}, we introduce some preliminary lemmas.
In Section \ref{sec:goodness}, we establish the nested goodness lemma.
In Section \ref{sec:mainproof}, we prove  Theorem~\ref{thm:main}.
In Section \ref{sec:conclusion}, we give some concluding remarks.

\section{Preliminary lemmas } \label{sec:lemmas}

In this section, we first give some useful lemmas. They will be used in motivating
some definitions and will also be used in the proofs in later sections. 
We start with a standard convexity-based inequality, which is sometimes referred to as the power means
inequality. We include a proof for completeness.

\begin{lemma} \label{lem:power-means}
Let $n$ be a positive integer. Let $1\leq a\leq b$ be reals.
Let $x_1,\dots, x_n$ be reals. Then
\[\sum_{i=1}^n x_i^a\leq n^{1-a/b} \cdot\left(\sum_{i=1}^n x_i^b\right)^{a/b}.\]
Equivalently, 
\[\left(\frac{1}{n} \sum_{i=1}^n x_i^a\right)^{1/a}\leq \left(\frac{1}{n}\sum_{i=1}^n x_i^b \right)^{1/b}.\]
\end{lemma}
\begin{proof}
Since the function $x^{b/a}$ is
either linear or is concave up, by Jensen's inequality, we have
$\sum_{i=1}^n x_i^b=\sum_{i=1}^n (x_i^a)^{b/a}\geq n [\frac{1}{n}\sum_{i=1}^n x_i^a]^{b/a}$. Rearranging, we obtain the desired inequalities.
\end{proof}

Let $G$ be a graph with $n$ vertices. Let $r$ be a positive integer. 
Recall that $p_r(G):=t_{K_{1,r}}(G)^{1/r}$.

\begin{lemma} \label{lem:r-norm}
For any graph $G$ and positive integers $r$, $p_r(G)=\frac{1}{n}\left(\frac{1}{n}\sum_{v\in V(G)} d(v)^r\right)^{1/r}$.
\end{lemma}
\begin{proof}
Let $n=|G|$. Recall that $t_{1,r}(G)=h_{K_{1,r}(G)}/n^{r+1}$,
where $h_{K_{1,r}}(G)$ is the number of homomorphisms from $K_{1,r}$ to $G$.
It is easy to see that $h_{K_{1,r}}(G)=\sum_{v\in V(G)} d(v)^r$.
Hence, $p_r(G)=t_{K_{1,r}}(G)^{1/r}=\left(\frac{1}{n^{r+1}}\sum_{v\in V(G)} d(v)^r\right)^{1/r}=\frac{1}{n}
\left(\frac{1}{n}\sum_{v\in V(G)} d(v)^r\right)^{1/r}$.
\end{proof}
Lemma~\ref{lem:r-norm} and Lemma~\ref{lem:power-means} imply the following useful fact.
\begin{lemma} \label{lem:monotone}
For any graph $G$ and positive integers $r\geq s$, we have 
$p_r(G)\geq p_s(G)$.
\end{lemma}

\section{Nested goodness Lemma} \label{sec:goodness}

Given a set $W$ and a sequence $S$ of elements of $W$,
we call $S$ a sequence in $W$ for brevity. The length of
$S$ is defined to be number of elements in the sequence $S$
(multiplicity counted) and is denoted by $|S|$. Given a positive integer $k$,
we let $W^k$ denote the set of sequences of length $k$ in $W$ and
we let $W_k$ denote the set of sequences of length $k$ in $W$ in which
the $k$ elements are all different.
Given a graph $G$ and a sequence $S$ in $V(G)$, the {\it common neighborhood} $N(S)$ is the set of vertices adjacent
to every vertex in $S$.

We now introduce a goodness notion that is inspired by Lemma 2.1 of \cite{CFS}.
A more specialized version of it was introduced in \cite{JN}.

\begin{definition}[$i$-good sequences] \label{def:good}
Let $0<\alpha,\beta<1$ be reals. Let $h,r$ be positive integers.
Let $G$ be an $n$-vertex graph.
Let $p=p_r(G)=t_{K_{1,r}}(G)^{1/r}$. For each $0\leq i\leq h$,
we define an $i$-good sequence in $V(G)$ relative to $(\al,\beta,h,r)$ (or simply {\it $i$-good} in short) 
as follows. We say that a sequence $T$ in $V(G)$ is {\it $0$-good} if
$|N(T)|\geq \alpha p^{|T|} n$.
For all $1\leq i\leq h$, we say that a sequence $S$ of length at most $h$ in $V(G)$ is {\it $i$-good} if $S$ is $0$-good and for each $|S|\leq k\leq h$, the number of $(i-1)$-good sequences of length $k$ in $N(S)$ is at least
$(1-\beta) |N(S)|^k$.
\end{definition}

Below is our main theorem on the goodness notion.

\begin{theorem}[Nested goodness lemma]\label{thm:universal}
Let $h\geq r$ be positive integers. Let $0<\beta<1$ be a real.
There is a positive real $\al$ depending on $h,r$ and $\beta$ such that the following is true. 
Let $G$ be any graph on $n$ vertices.
Let $p=p_r(G)=t_{K_{1,r}}(G)^{1/r}$. For any $i,j\in [h]$, 
let $\A_{i,j}$ denote the set of $i$-good sequences of length $j$ relative to
$(\al,\beta, h,r)$ in $V(G)$. Then for each $i\in [h]$ and $r\leq j\leq h$
\[\sum_{S\in\A_{i,j}} |N(S)|^r\geq (1-\beta) n^{j+r} p^{jr}.\]
In particular, there exists an $i$-good sequence $S$ of size $j$  such that 
$|N(S)|\geq (1-\beta)^{1/r}p^j n$.
\end{theorem}

Applying Theorem~\ref{thm:universal} with $r=1$, we get the following theorem that
is of independent interest.

\begin{theorem}\label{thm:goodness}
Let $h$ be a positive integer. Let $0<\beta<1$ be a real.
There is a positive real $\al$ depending on $h$ and $\beta$ such that the following is true. 
Let $G$ be any graph on $n$ vertices.
Let $p=\frac{2e(G)}{n^2}$. For any $i,j\in [h]$, 
let $\A_{i,j}$ denote the set of $i$-good sequences of length $j$ relative to
$(\al,\beta, h,1)$ in $V(G)$. Then, for all $i \in [h]$
\[\sum_{S\in\A_{i,j}} |N(S)|\geq (1-\beta) n^{j+1} p^j.\]
In particular, there exists an $i$-good sequence $S$ of size $j$  such that 
$|N(S)|\geq (1-\beta)p^j n$.
\end{theorem}

Loosely speaking, one may think of the usual dependent random choice lemma as saying that for any 
positive integers $j, h$ and real $0<\beta<1$,
there is a $1$-good sequence $S$ of size $j$ relative to $(\al,\beta,h,1)$ 
for some appropriate $\al>0$ such that  most of the subsets $T$ in $N(S)$
of size at most $h$ have their common neighborhood fractionally as large as 
expected in the random graph of
the same edge-density. In that regard, one may view
Theorem~\ref{thm:goodness} as a strengthening of the dependent random choice lemma
to a stronger notion of goodness.
Theorem~\ref{thm:universal} follows from the following more technical lemma.

\begin{lemma} \label{lem:goodness2}
Let $h\geq r$ be positive integers. Let $0<\beta<1$ be a real.
There exists a positive real $\alpha$ depending on $h,r$ and $\beta$ such that
the following is true. Let $G$ be a graph on $n$ vertices. 
Let $G$ be any graph on $n$ vertices.
Let $p=p_r(G)=t_{K_{1,r}}(G)^{1/r}$. 
For each $0\leq i\leq h$ and $1\leq j\leq h$,
let $\A_{i,j}$ denote the set of $i$-good sequences of length $j$ relative to
$(\al,\beta, h,r)$ in $V(G)$
and let $\B_{i,j}=[V(G)]^j\setminus \A_{i,j}$.
Then for each $0\leq i\leq h, 1\leq j  \leq h$ and $1\leq \ell\leq j$, 
\[\sum_{S\in \B_{i,j}} |N(S)|^\ell \leq \beta  n^{j+\ell}p^{j\ell}.\]
\end{lemma}
\begin{proof}
Suppose $\al$ has been specified, we define a sequence $\al_i$, $0\leq i\leq h$, by letting $\al_0=\alpha$ and
$\al_i=\al+h(\al_{i-1}/\beta)^{1/h}$ for each $i\in [h]$. For fixed $h$ and $\beta$, it is easy to 
see that by choosing $\al$ to be small enough, we can ensure that $\al_i$ in increasing in $i$
and $\al_h<\beta$. Let us fix such an $\al$. Now, let $\A_{i,j}$ and $\B_{i,j}$ be defined as stated.
We use induction on $i$ to prove that 
for all $0\leq i  \leq h, j\in [h]$, and $1\leq \ell\leq j$
\[\sum_{S\in \B_{i,j}} |N(S)|^\ell \leq \al_i n^{j+\ell}p^{j\ell}.\]
For the basis step, let $i=0$.
Let $j,\ell\in [h]$ where $\ell\leq j$.
By definition, 
\begin{equation} \label{eq:Dj-bound}
\sum_{S\in \B_{0,j}} |N(S)|^\ell \leq n^j (\alpha p^jn)^\ell\leq \alpha n^{j+\ell}p^{j\ell}\leq \al_0  n^{j+\ell}p^{j\ell}.
\end{equation}
Hence the claim holds for $i=0$.
For the induction step, let $i\geq 1$ and suppose the claims hold when $i$ is replaced with $i-1$. Let $j\in [h]$.
For each $j\leq k\leq h$, let $\C^k_{i,j}$ denote
the set of sequences $S$ in $\B_{i,j}$ such that the number of
sequences of length $k$ in $N(S)$ that are not $(i-1)$-good is at least $\beta|N(S)|^k$.
By definition, $\B_{i,j}=\B_{0,j}\cup \bigcup_{k=j}^h \C^k_{i,j}$.
Let  $\F_k$ be the collection of pairs $(S,T)$, where $S\in \C^k_{i,j}$
and $T$ is a sequence of length $k$ in $N(S)$ that is not $(i-1)$-good. By our definition,
\[|\F_k|\geq \sum_{S\in \C^k_{i,j}} \beta |N(S)|^k=\beta \cdot \sum_{S\in \C^k_{i,j}} |N(S)|^k.\]
On the other hand, for each sequence $T$ of length $k$ in $V(G)$ that is
not $(i-1)$-good, the number of sequences $S$ of length $j$ that satisfy
$(S,T)\in \F_k$ is most $|N(T)|^j$. Hence,
\[|\F_k|\leq \sum_{T\in \B_{i-1,k}} |N(T)|^j\leq\al_{i-1} n^{j+k}p^{jk},\]
where the last inequality follows from the induction hypothesis.
Combining the lower and upper bounds on $|\F_k|$, we get
\begin{equation} \label{eq:Ckij}
   \sum_{S\in \C^k_{i,j}} |N(S)|^k\leq (\al_{i-1}/\beta) n^{j+k} p^{jk}. 
\end{equation}
Let $\ell\in [h]$ such that $\ell\leq j$. Since $j\leq k$,  we have $\ell\leq k$.
Applying Lemma~\ref{lem:power-means} with $a=\ell,b=k$ and using $|\C^k_{i,j}|\leq n^j$, we
get
\[\sum_{S\in \C^k_{i,j}} |N(S)|^\ell\leq (n^j)^{1-\ell/k}(\al_{i-1}/\beta)^{\ell/k} (n^{j+k}p^{jk})^{\ell/k}
\leq (\al_{i-1}/\beta)^{1/h} n^{j+\ell}p^{j\ell},\]
where we used the fact that $\al_{i-1}/\beta<1$.
By ~\eqref{eq:Dj-bound} and ~\eqref{eq:Ckij}, we have
\[\sum_{S\in \B_{i,j}} |N(S)|^\ell\leq \sum_{S\in \B_{0,j}} |N(S)|^\ell+\sum_{k=j}^h \sum_{S\in \C^k_{i,j}} |N(S)|^\ell 
\leq [\alpha+ h(\al_{i-1}/\beta)^{1/h}] n^{j+\ell} p^{j\ell}\leq \al_i n^{j+\ell} p^{j\ell} .\] 
This completes the induction and the proof.
\end{proof}

We need another quick lemma. Given two positive integers $n,j$, let $n_j=n(n-1)\cdots (n-j+1)$.

\begin{lemma} \label{lem:total-sequence}
Let $G$ be a graph on $n$ vertices and $j,r$
positive integers.
Let $p=p_r(G)=t_{K_{1,r}}(G)^{1/r}$. Then
$\sum_{S\in [V(G)]^j} |N(S)|^r\geq n^{j+r}p^{jr}$.
If $p>4jn^{-1/r}$ then
$\sum_{S\in [V(G)]_j} |N(S)|^r\geq \frac{1}{2^{j+1}}n^{j+r}p^{jr}$.
\end{lemma}
\begin{proof}
First, note that $\sum_{T\in [V(G)]^r} |N(T)|=h_{K_{1,r}}(G)=nt_{K_{1,r}}(G)=n^{r+1}p^r$. Hence, by convexity
\[ \sum_{T\in [V(G)]^r} |N(T)|^j
\geq n^r \left(\frac{\sum_{T\in [V(G)]^r} |N(T)|}{n^r}\right)^j
=n^r(np^r)^j =n^{j+r}p^{jr}.\]
If $p>4jn^{-1/r}$, then 
\[\sum_{T\in [V(G)]^r,|N(T)|\geq 2j} |N(T)|^j\geq n^{j+r}p^{jr}-n^r(2j)^j\geq \frac{1}{2} n^{j+r}p^{jr}. \]
Hence,
\[\sum_{T\in [V(G)]^r,|N(T)|\geq 2j} |N(T)|_j\geq 
\sum_{T\in [V(G)]^r,|N(T)|\geq 2j} (|N(T)|/2)^j\geq \frac{1}{2^{j+1}} n^{j+r}p^{jr}. \]
To prove the first statement, note that $\sum_{S\in [V(G)]^j} |N(S)|^r$ counts pairs $(S,T)$, where $S$ is
a sequence of length $j$ and $T$ is a sequence of length $r$ in $N(S)$.
By double counting, we have $\sum_{S\in [V(G)]^j} |N(S)|^r=\sum_{T\in [V(G)]^r} |N(T)|^j
\geq n^{j+r}p^{jr}$.

For the second statement, note that $\sum_{S\in [V(G)]_j} |N(S)|^r$ counts pairs
$(S,T)$, where $S$ is a sequence of length $j$ with no repetition and $T$ is sequence
of length $r$ in $N(S)$. By double counting and convexity, we have
$\sum_{S\in [V(G)]_j} |N(S)|^r =\sum_{T\in [V(G)]^r} |N(T)|_j\geq  \frac{1}{2^{j+1}} n^{j+r}p^{jr}$.
\end{proof}

Now we are ready to prove Theorem~\ref{thm:universal}.

\medskip

{\bf Proof of Theorem~\ref{thm:universal}:}
Let $h\geq r$ be positive integers and $0<\beta<1$ a real. Let $\al$ be defined as in Lemma~\ref{lem:goodness2}.
Let $i\in [h]$ and $r\leq j\leq h$. Let $\A_{i,j}$ denote the set of $i$-good sequences of length $j$ relative to
$(\al,\beta, h,r)$ in $V(G)$ and let $\B_{i,j}=[V(G)]^j\setminus \A_{i,j}$. 
By Lemma~\ref{lem:total-sequence}, $\sum_{S\in [V(G)]^j} |N(S)|^r\geq n^{j+r}p^{jr}$. 
By Lemma~\ref{lem:goodness2}, $\sum_{S\in \B_{i,j}} |N(S)|^r \leq \beta n^{j+r}p^{jr}$.
Hence, $\sum_{S\in \A_{i,j}} |N(S)|^r\geq (1-\beta) n^{j+r}p^{jr}$, as desired.
This proves the first part of the theorem.
Now, since $|\A_{i,j}|\leq n^j$, by averaging, there exists an $S\in \A_{i,j}$ such that
$|N(S)|^r\geq (1-\beta)p^{jr}$ and hence $|N(S)|\geq (1-\beta)^{1/r}p^j n$. This proves
the second part of the theorem.
\hfill $\Box$

In order to prove the second part of Theorem~\ref{thm:main}, we need the following variant of Theorem~\ref{thm:goodness}.
We omit the proof since it is almost identical to that of Theorem~\ref{thm:goodness}, except that we use the second statement of
Lemma~\ref{lem:total-sequence} instead of the first statement. 

\begin{lemma} \label{lem:goodness-distinct}
Let $h\geq r$ be positive integers. Let $0<\beta<1$ be a real.
There is a positive real $\al$ depending on $h,r$ and $\beta$ such that the following is true. 
Let $G$ be any graph on $n$ vertices.
Let $p=p_r(G)=t_{K_{1,r}}(G)^{1/r}$. For any $i,j\in [h]$, 
let $\A^*_{i,j}$ denote the set of $i$-good sequences of length $j$ relative to
$(\al,\beta, h,r)$ in $V(G)$ that has no repetition. 
If $p>4jn^{-1/r}$, then for each $i\in [h]$ and $r\leq j\leq h$
\[\sum_{S\in\A^*_{i,j}} |N(S)|^r\geq (\frac{1}{2^{j+1}}-\beta) n^{j+r} p^{jr}.\]
\end{lemma}

\section{Proof of Theorem~\ref{thm:main}} \label{sec:mainproof}

{\bf Proof of Theorem~\ref{thm:main}:}
Let $n=|G|$ and $h=|H|$.  Let $\beta=\frac{1}{2^{h+2}}$. 
Let $\al$ be the positive constant
given by Theorem~\ref{thm:universal} for the given $h$, $\beta$, and $r$.
Let 
\[c_1=c_1(H)=(1-\beta)^{|B_1|/|B_0|}\al^{\sum_{i=2}^m |B_i|} (1-h\beta)^{m-1}.\]
Let 
\[ c_2=c_2(H)= 4h/\al, \mbox{ and } c_3=c_3(H)=c_1/2^{h^2}.\]

Suppose $H$ has root block $B_0$ and non-root blocks $B_1,\dots, B_m$ such that
$B_1$ is completely joined to its parent set $P(B_1)=B_0$ and for each $i=2,\dots, m$, $B_i$ is
completely joined to its parent set $P(B_i)$ where $P(B_i)\subseteq B_{\gamma(i)}$ for some $1\leq \gamma(i)<i$
and $|P(B_i)|\geq |P(B_{\gamma(i)})|$.
For each $i\in [m]$, let $\F_i$ denote the collection of
all the parent sets $P(B_j)$ that are contained in $B_i$.
Let $T$ be a tree with $V(T):=\{v_0,v_1,\dots, v_m\}$ and edge set
$E(T):=v_0v_1\cup \{v_iv_{\gamma(i)}: \in [m]\}$. We call $T$ the {\it auxiliary tree} for $H$. For each $i\in [m]$, define the {\it depth} of $B_i$, denoted by
$d_i$, to be the distance from $v_0$ to $v_i$ in $T$. Let $q$ denote the
maximum depth of a block. Then clearly $q\leq m\leq h-1$. 

Let $G$ be any graph.
For convenience, we say that a sequence in $V(G)$ is $i$-good if it is $i$-good relative to $(\al,\beta,h,r)$.
As in Theorem~\ref{thm:universal},
for each $0\leq i\leq h$ and $r\leq j\leq h$,
let $\A_{i,j}$ be the set of $i$-good sequences of length $j$ 
in $V(G)$. Let $\B_{i,j}=[V(G)]^j\setminus \A_{i,j}$.
Let $\A^*_{i,j}$ be the set of $i$-good sequences of length $j$ in $V(G)$
that contains no repetition. Let $f$ be the uniform random mapping from $V(H)$ to $V(G)$.

Let 
\begin{eqnarray*}
E_1 &=& \mbox{ the event that } f(B_0)\in \A_{q,|B_0|} \mbox{ and }
f(B_1)\in [N(f(B_0))]^{|B_1|},\\
F_1&=&\mbox{ the event that each sequence in } \F_1 \mbox{ is mapped to a $(q-1)$-good sequence },\\
E^*_1&=& \mbox{ the event that } f(B_0)\in \A^*_{q,|B_0|} \mbox{ and }
f(B_1)\in [N(f(B_0))]^{|B_1|}.
\end{eqnarray*}
For each $i\in \{2,\dots, m\}$, let 
\begin{eqnarray*}
E_i &=& \mbox{ the event  that } f(B_i)\in [N(f(P(B_i))]^{|B_i|},\\
F_i&=&\mbox{ the event that each sequence in $\F_i$ is mapped to an $(q-d_i)$-good sequence} \\
L_i&=&\mbox{ the event that } f \mbox{ is injective on } B_0\cup B_1\cup \dots \cup B_i.
\end{eqnarray*}

Recall that $s=|B_0|$ and $r=\max_i |P(B_i)|$.
By Theorem~\ref{thm:universal},
\begin{equation} \label{eq:B0-embed}
    \sum_{S\in \A_{q,|B_0|}} |N(S)|^s\geq (1-\beta) n^{2s}p^{s^2}.
\end{equation}
Furthermore, by Lemma ~\ref{lem:goodness-distinct} if $p\geq 4jn^{-1/r}$ then 
\begin{equation} \label{eq:B0-embed2}
    \sum_{S\in \A^*_{q,|B_0|}} |N(S)|^s\geq (\frac{1}{2^{j+1}}-\beta) n^{2s}p^{s^2}.
\end{equation}

Hence, since $|B_1|\geq |B_0|=s$, using \eqref{eq:B0-embed} and convexity we get
\[
\sum_{S\in \A_{q,|B_0|}} |N(S)|^{|B_1|}
=\sum_{S\in \A_{q,|B_0|}} (|N(S)|^s)^{\frac{|B_1|}{s}}\geq
n^s (\frac{1}{n^s} (1-\beta) n^{2s} p^{s^2})^\frac{|B_1|}{s}
= (1-\beta)^{\frac{|B_1|}{|B_0|}} n^{|B_0|+|B_1|}p^{|B_0||B_1|},
\]
and if $p\geq 4jn^{-1/r}$ then 
\[
\sum_{S\in \A^*_{q,|B_0|}} |N(S)|^{|B_1|}
= (\frac{1}{2^{j+1}}-\beta)^{\frac{|B_1|}{|B_0|}} n^{|B_0|+|B_1|}p^{|B_0||B_1|},
\]
Hence,

\begin{equation} \label{eq:E1-probab}
\mathbb{P}(E_1)=\sum_{S\in \A_{q,|B_0|}} \frac{1}{n^{|B_0|}}\cdot 
\frac{|N(S)|^{|B_1|}}{n^{B_1}} 
=\frac{1}{n^{|B_0|+|B_1|}} \sum_{S\in \A_{q,|B_0|}} |N(S)|^{|B_1|}
\geq (1-\beta)^{\frac{|B_1|}{|B_0|}} p^{|B_0||B_1|},
\end{equation}
and if $p\geq 4jn^{-1/r}$ then 

\begin{equation} \label{eq:E1-probab2}
\mathbb{P}(E^*_1)=\sum_{S\in \A^*_{q,|B_0|}} \frac{1}{n^{|B_0|}}\cdot 
\frac{|N(S)|^{|B_1|}}{n^{B_1}} 
\geq (\frac{1}{2^{j+1}}-\beta)^{\frac{|B_1|}{|B_0|}} p^{|B_0||B_1|}
\geq  (\frac{1}{2^{h+2}})^{|B_1|}p^{|B_0||B_1|}.
\end{equation}

We now bound $\mathbb{P}(F_1|E_1)$. Recall that $\F_1$ consists of
parent sets $P(B_j)$ that are contained in $B_1$. By requirement,
these sets have size at least $|P(B_1)|=|B_0|$.
Let $S$ be any fixed sequence in $\A_{q,|B_0|}$.
By the definition of $\A_{q,|B_0|}$, for each $|B_0|\leq k\leq h$, the number of $(q-1)$-good sequences of length $k$ in $N(S)$
is at least $(1-\beta) |N(S)|^k$. So, conditioning on $f$ mapping $B_0$ to $S$ and $B_1$ to $N(S)$,
the probability that $f$ maps any particular sequence in $\F_1$ to an $(q-1)$-good sequence is
at least $(1-\beta)$. Since there are clearly at most $h$ sequences in $\F_1$, the probably
that $f$ maps every sequence in $\F_1$ to a $(q-1)$-good sequence is at least $1-h\beta$.
Hence
\begin{equation} \label{eq:F1-probab}
\mathbb{P}(F_1|E_1)\geq 1-h\beta.
\end{equation}

For each $i=2,\dots, h$, we estimate $\mathbb{P}(E_i|E_1F_1\dots E_{i-1}F_{i-1})$.
Assume the event $E_1F_1\cdots E_{i-1}F_{i-1}$. Since $P(B_i)\subseteq B_{\gamma(i)}$,
where $\gamma(i)<i$, by our assumption, $P(B_i)$ is mapped to a $(q-d_{\gamma(i)})$-good 
sequence. Since a $(q-d_{\gamma(i)})$-sequence is $0$-good by definition, we have $|N(f(P(B_i)))|\geq \al p^{|P(B_i)|} n$.
Hence,
\begin{equation}\label{eq:Ei-probab}
\mathbb{P}(E_i|E_1F_1\dots E_{i-1}F_{i-1})=\frac{|N(f(P(B_i))|^{B_i|}}{n^{|B_i|}}\geq 
\frac{(\al p^{|P(B_i)|} n)^{|B_i|}}{n^{|B_i|}}
=\al^{|B_i|} p^{|P(B_i)||B_i|}.
\end{equation}
Now assume $E_1F_1\dots E_{i-1}F_{i-1}E_i$. Since $S:=f(P(B_i))$ is a $(q-d_{\gamma(i)})$-good
sequence, by definition, for each $|S|\leq k\leq h$ the number of $(q-1-d_{\gamma(i)})$-good sequences
is at least $(1-\beta)|N(S)|^k$. Since there are at most $h$ sequences in $\F_i$, as in deriving 
\eqref{eq:F1-probab}, we have
\begin{equation} \label{eq:Fi-probab}
\mathbb{P}(F_i|E_1F_1\dots E_{i-1}F_{i-1}E_i) \geq 1-h\beta.
\end{equation}
By ~\eqref{eq:E1-probab},~\eqref{eq:F1-probab},~\eqref{eq:Ei-probab}, and ~\eqref{eq:Fi-probab}, 
\begin{align} \label{eq:f-hom}
\mathbb{P}(\mbox{$f$ is a homomorphism}) &\geq \mathbb{P}(E_1F_1\dots E_{m-1}F_{m-1}E_m) \nonumber\\
&\geq (1-\beta)^{\frac{|B_1|}{|B_0|}}\al^{\sum_{i=2}^m |B_i|} (1-h\beta)^{m-1}p^{|B_0||B_1|+\sum_{i=2}^m |P(B_i)||B_i|}\\ 
&= c_1 p^{e(H)}.\nonumber
\end{align}
The proves the first and the main part of the theorem. 

For the second statement, suppose $h_{K_{1,r}}(G)>c_2n^r$. Then 
\[p=p_r=(h_{K_{1,r}}/n^{r+1})^{1/r}\geq c_2^{1/r} n^{-1/r}\geq (4h/\al)^{1/r}n^{-1/r}.\]
For each $i\geq 2$, we bound $\mathbb{P}(L_i|E^*_1F_1E_2F_2L_2\dots L_{i-1}E_iF_i)$.
Assume $E^*_1F_1E_2F_2L_2\dots L_{i-1}E_iF_i$. By our assumption $P(B_i)$ is mapped to
a $(q-d_{\gamma(i)})$-good sequence and $B_i$ is mapped into $N(f(P(B_i))$.
Since $f(P(B_i))$ is $0$-good, $|N(f(P(B_i))|\geq \al p^{|P(B_i)|} n\geq \al [(4h/\al)^{1/r} n^{-1/r}]^r n=4h$,
where we used the fact that $|P(B_i)|\leq r$. Given $E^*_1F_1E_2F_2L_2\dots L_{i-1}F_{i-1}$, the probability
that $f$ maps $B_i$ injectively into $N(F(P(B_i))$ and avoids $f(B_0\cup B_1\cup\dots\cup B_{i-1})$ is at least
$(3h)_{|B_i|}/(4h)^{|B_i|}>(1/2)^{|B_i|}$. Hence, 
\begin{equation} \label{eq:Li-bound}
\mathbb{P}(L_i|E^*_1F_1E_2F_2L_2\dots L_{i-1}E_iF_i)>(1/2)^{|B_i|}.
\end{equation}
By ~\eqref{eq:E1-probab2}, ~\eqref{eq:Li-bound}, and a similar calculation as in ~\eqref{eq:f-hom}, we have
\begin{eqnarray*}
\mathbb{P}(f \mbox{ is an injective homomorphism})&\geq& \mathbb{P}(E^*_1F_1E_2F_2L_2\dots E_mF_mL_m) \\
&\geq&  (\frac{1}{2^{h+2}})^{|B_1|}(\frac{1}{2})^{|B_2|+\cdots+|B_m|} c_1 p^{e(H)}\\
&\geq& \frac{1}{2^{h^2}} c_1  p^{e(H)} =c_3 p^{e(H)}.
\end{eqnarray*}
This proves the second part of the theorem. \hfill $\Box$


\section{Concluding remarks} \label{sec:conclusion}

\medskip

In this note, we used a nested variant of the dependent random choice to not only embed
an appropriate tree-degenerate bipartite graph $H$ in a host graph $G$, but also give tight (up to a multiplicative factor)
counting bound on the number of copies of $H$ in $G$. In this variant, we get extra goodness features
almost for free. It will be interesting to find more applications of it.

Another interesting feature of Theorem~\ref{thm:main} is that the condition of the host graph is relaxed from $1$-norm density
to $r$-norm density, which makes the result more flexible for applications. In principle, one could
study the so-called {\it r-norm} Tur\'an problem for bipartite graphs, where one wants to determine the maximum $r$-norm density of
an $H$-free graph on $n$-vertices for a given bipartite graph $H$. The problem seems particularly natural for the family
of $r$-degenerate graphs. For hypergraph co-degree problems, such a study has recently been initiated
by Balogh, Clemen, and Lidick\'y \cite{BCL, BCL2}.

Last but not least, it will be highly desirable to  make more progress on Conjecture~\ref{degenerate-conjecture} beyond
the following general bound obtained by Alon, Krivelevich, and Sudakov \cite{AKS}, 
which has stood as the best known bound in the last two decades.
\begin{theorem}[\cite{AKS}]
If $H$ is an $r$-degenerate bipartite graph, then $\ex(n,H)=O(n^{2-1/4r})$.
\end{theorem}

\end{document}